\documentclass[final,reqno]{amsart}

\usepackage{graphicx}
\usepackage{amsmath,amsthm,amsfonts,amscd,amssymb}
\usepackage[color,notref,notcite]{showkeys}
\usepackage{url}
\usepackage{subeqnarray}
\usepackage[color,notref,notcite]{showkeys} 
\definecolor{labelkey}{rgb}{0.6,0,1}
\usepackage{enumitem}
\usepackage{algorithm}
\usepackage{algpseudocode}
\usepackage{citesort}

\theoremstyle{plain}
\newtheorem{theorem}{Theorem}[section]

\theoremstyle{definition}
\newtheorem{definition}[theorem]{Definition}
\newtheorem{test}[theorem]{Test}
\def\bhyp#1{\begin{equation}\label{#1}\begin{array}{c}}
\def\ehyp{\end{array}\end{equation}}

\newcounter{cst}

\theoremstyle{remark}
\newtheorem{remark}[theorem]{Remark}

\numberwithin{equation}{section}
\numberwithin{figure}{section}

\newcommand{\RR}{{\mathbb R}}

\def\O{\Omega}
\def\dsp{\displaystyle}

\def\disc{{\mathcal D}}
\def\mesh{{\mathcal M}}

\def\edges{{\mathcal E}}
\def\edge{\sigma}

\newcommand{\edgescv}{{{\edges}_K}}  

\def\dr{\partial}
\newcommand{\centeredge}{\overline{x}_\edge} 

\newcommand{\cK}{{\mathcal K}}

\newcommand{\x}{\pmb{x}}

\newcommand{\cA}{{\mathcal A}}
\newcommand{\cB}{{\mathcal B}}

\newif\ifcorr\corrtrue
\usepackage[normalem]{ulem}
\definecolor{violet}{rgb}{0.580,0.,0.827}

\def\bo{{\boldsymbol \omega}}

\newcommand{\ud}{\, \mathrm{d}} 
\def\div{\mathop{\rm div}}
\def\ini{\mathop{\rm ini}}

\title{A General Error Estimate For Parabolic Variational Inequalities}

\author{Yahya Alnashri}
\address[Yahya Alnashri]{Department of Mathematics, Al-Qunfudah University College, Umm Al-Qura University, Saudi Arabia}
\email{yanashri@uqu.edu.sa}


\date{\today}

\begin{document}
\newcommand{\subscript}[2]{$#1 _ #2$}

\begin{abstract}
The gradient discretisation method (GDM) is a generic framework designed recently, as a discretise in spatial space, to partial differential equations. This paper aims to use the GDM to establish a first general error estimate for numerical approximations of parabolic obstacle problems. This gives the convergence rates of several well--known conforming and non conforming numerical methods. Numerical experiments based on the hybrid finite volume method are provided to verify the theoretical results.      
\end{abstract}

\maketitle


\section{Introduction}
\par Parabolic variational inequalities (PVIs) appear in different applications in porous media and physic. Specifically, semipermeable membrane including osmosis phenomenon and problems concerning the control of temperature at thermal boundaries \cite{N-21-Louis,A4}. They may also be employed for studying a problem arising in financial mathematics \cite{psi-finance}. We consider in this paper a parabolic obstacle problem,
\begin{subequations}\label{pvi-obs}
\begin{align}
(\partial_t \bar u-\div(\Lambda(\x)\nabla\bar u)-f)(\bar u-\psi)=0 &\mbox{\quad in $\Omega\times(0,T)$,} \label{pvi-obs1}\\
\partial_t \bar u-\div(\Lambda(\x)\nabla\bar u)\geq f &\mbox{\quad in $\Omega\times(0,T)$,} \label{pvi-obs2}\\
\bar u \geq \psi &\mbox{\quad in $\Omega\times(0,T)$,} \label{pvi-obs3}\\
\bar{u} =0 &\mbox{\quad on $\partial\O\times(0,T)$,} \label{pvi-obs4}\\
\bar u(\x,0)=u_{\rm ini}&\mbox{\quad in $\Omega\times\{0\}$},
\end{align}
\end{subequations}

In what follows, let $[0,T]\subset \RR$ and $\O \subset \RR^d$ $(d=1,2,3)$ be a bounded connected open set. The assumptions on the data in Problem \eqref{pvi-obs} are the following:
\begin{equation}\label{asm-2}
\begin{aligned}
&\mbox{\textbullet\ the domain $\O$ has a Lipschitz boundary and $T>0$},\\
&\mbox{\textbullet\ $\Lambda:\O \to \mathbb M_d(\RR)$ is a measurable function ($\mathbb M_d(\RR)$ is the set of $d \times d$ matrices)}\\
&\exists\; \underline \lambda, \overline \lambda >0 \mbox{ s.t. for a.e. }\x \in \O,
\mbox{ $\Lambda(\x)$ is symmetric with eigenvalues in $[\underline \lambda, \overline \lambda]$, }\\
&\mbox{\textbullet\ the initial solution $u_{\rm ini} \in L^2(\O)$},\\
& \mbox{\textbullet\ the function $f\in L^2(\O\times(0,T))$ and the obstacle function $\psi \in L^2(\O)$}.
\end{aligned}
\end{equation}

The closed convex independent and dependent time sets are defined by
\begin{subequations}\label{convex-set}
\begin{align}
&\cK:=\{ v \in H_0^1(\O)\; : \; v(\x) \geq \psi(\x)\; \mbox{ for a.e. } \x \in \O \},\\
&\mathbb K:=\{ v\in L^2(0,T;H_0^1(\O)):\; v(t)\in \cK \mbox{ for a.e. } t\in [0,T] \}.
\end{align}
\end{subequations}

Indeed the set $\cK$ contains $\psi^+$ (it belongs to $H_0^1(\O)$) and thus the set $\mathbb K$ contains the constant in time function $t \mapsto \psi^+$. Under the above assumptions, the weak solution to \eqref{pvi-obs} is  
\begin{equation}\label{obs-lineaer}
\left\{
\begin{array}{ll}
\dsp\mbox{find } \bar u\in \mathbb{K}\cap C^0([0,T];L^2(\O)),\; \bar u(\cdot,0)=u_{\rm ini},\; \dr_t \bar u\in L^2(0,T;L^2(\O)) \mbox{ and }\\
\dsp\int_0^T\dsp\int_\O \partial_t \bar u(\x,t) (u(\x,t)-v(\x,t)) \ud \x \ud t\\
+\dsp\int_0^T\int_\O \Lambda(\x)\nabla\bar u(\x,t) \cdot \nabla(\bar u-v)(\x,t)\ud \x \ud t\\
[1em]
\leq\dsp\int_0^T \int_\O f(\x,t)(\bar u(\x,t)-v(\x,t))\ud \x \ud t, \quad \mbox{ for all } v\in \mathbb K.
\end{array}
\right.
\end{equation}

\par The theoretical results concerning the existence and uniqueness of the solution, as well as its regularity, to the linear PVIs can be found in foundational studies such as \cite{ex-pvi-soln,ex-pvi-soln-2,reg-pvi-soln,reg-pvi-soln2}. An overview of numerical studies for PVIs be found in Glowinski's et al. monograph \cite{G1}.

\par A number of numerical methods have been proposed for approximation of variational inequality problems; finite elements, finite difference methods, finite volume methods, Monte Carlo methods, and non conforming finite elements methods among them. In \cite{P7}, a discretisation framework, using the backwards Euler and Galerkin methods, is developed for the parabolic Signorini problem. Regarding $\mathbb P1$ finite element methods for parabolic obstacle problem, we refer the reader to \cite{P6,P10,P16,P8,P5,P1,P2}. The $L^\infty$--convergence and the error estimate for $\mathbb P1$ finite element method, applied on triangular meshes with acute angles, are obtained in \cite{P5} under regularity assumptions on the solution ($\partial _t \bar u$ and $\Delta \bar u$ in $L^\infty(\O\times(0,T)))$. \cite{P1} provides a posteriori error estimate of order $\mathcal O (h+\tau)$, where $\tau$ is the time step, for linear finite element method for the parabolic obstacle problem, provided that the initial solution $\bar u_0$ is smooth. \cite{1st-FEM-pvi} establishes the convergence analysis of a finite element method for PVIs with $\psi=0$ and it is generalised to a general function $\psi \in H^2(\O)$ in \cite{2nd-FEM-pvi}. Berton and Eymard \cite{FVM-pvi} use upwind implicit finite volume scheme to approximate PVIs motived by American options contract. 

\par There are a number of an posteriori and a priori error analysis available for the PVIs. A priori error estimates for linear finite element method is driven in \cite{N-pvi-4,N-pvi-20,N-pvi-34} with zero obstacle and the time derivative of continuous solution is a function in $L^2(0,T;H^1(\O))$. \cite{Crouzeix-Raviart-pvi} designs and analyses a Crouix--Raviart finite element method for the model with a non affine obstacle. In this analysis, an order $\mathcal O(h+\tau)$ for the $L^2(0,T;H^1(\O))$ error is proved under an assumption $\delta_t \bar u \in L^2(0,T;L^2(\O))$.

\par The goal of this contribution is to afford a first simpler general error estimate for the approximation of parabolic obstacle problem by conforming and non conforming methods. For this purpose, we consider using the gradient discretisation method (GDM) for the approximation of Model \eqref{pvi-obs}. The GDM is a generic tool to provide a unified numerical analysis to different spatial partial differential equations \cite{S1}. The analysis obtained by the GDM applies to various methods; conforming and non conforming finite elements, finite volumes, for instance.  

\par This paper is organised as follows. Section \ref{sec-DS} introduces the discrete elements and propose the discrete scheme for the evolution model \eqref{pvi-obs}. Section \ref{sec-linear} states and proves the novel results (Theorem \ref{theorem-error}) concerning general error estimates, which are new since they apply to several methods included in the GDM. Since we do not deal here with a specific scheme, the approach used in previous studies cannot be efficient in obtaining our results. Instead, we rely on our technique used in \cite{AD14} to deal with inequalities besides developing a similar technique as in \cite{D2017-PPDE}. Section \ref{test} is devoted to numerical experiments to demonstrate the generic analysis results. We perform the Hybrid Mimetic Mixed method on different types of meshes, including very distorted ones.

\section{Discrete Setting}\label{sec-DS}
We first introduce the basic discrete elements (called gradient discretisation), consisting of discrete space and operators. We then construct the approximate numerical scheme (called gradient scheme).

\begin{definition}[{\bf Gradient discretisation}]\label{def-gd-pvi-obs}
Let $\O$ be an open subset of $\RR^d$ (with $d=1,2,3$) and $T>0$. A space time gradient discretisation $\disc$ for the obstacle problem with homogeneous Dirichlet boundary conditions is a family $\disc=(X_{\disc,0},\Pi_\disc,\nabla_\disc,J_\disc, (t^{(n)})_{n=0,...,N})$, where:
\begin{enumerate}
\item The set $X_{\disc,0}$ of discrete unknowns is a finite-dimensional vector space over $\RR$, taking into account the zero boundary condition \eqref{pvi-obs4},
\item $\Pi_\disc : X_{\disc,0} \to L^2(\O)$ is a linear mapping, called the function reconstruction operator,
\item $\nabla_\disc : X_{\disc,0} \to L^2(\O)^d$ is a linear mapping, called the gradient reconstruction operator, and must be defined so that $\| \nabla_\disc\cdot ||_{L^2(\O)^d}$ a norm on $X_{\disc,0}$. 
\item $J_\disc: \cK \to \cK_\disc$ is a linear and continuous interpolation operator for the
 initial conditions, where $\cK_\disc:=\{ v\in X_{\disc,0} :\; \Pi_\disc v\geq \psi, \mbox{ for a.e. in } \O \}$ and $\cK$ is the set in which the continuous solution belongs to, 
\item $t^{(0)}=0<t^{(1)}<....<t^{(N)}=T$.
\end{enumerate}
\end{definition}

Let us introduce some notations to define the space--time reconstructions $\Pi_\disc v: \O\times[0,T]\to \RR$, and $\nabla_\disc v: \O\times[0,T]\to \RR^d$, and the discrete time derivative $\delta_\disc v : (0,T) \to L^2(\O)$, for $v=(v^{(n)})_{n=0,...,N} \in X_{\disc,0}^{N}$. 

For a.e $\x\in\O$, for all $n\in\{0,...,N-1 \}$ and for all $t\in (t^{(n)},t^{(n+1)}]$, let
\begin{equation*}
\begin{split}
&\Pi_\disc v(\x,0)=\Pi_\disc v^{(0)}(\x), \quad \Pi_\disc v(\x,t)=\Pi_\disc v^{(n+1)}(\x),\\
&\nabla_\disc v(\x,t)=\nabla_\disc v^{(n+1)}(\x).
\end{split}
\end{equation*}
Set $\delta t^{(n+\frac{1}{2})}=t^{(n+1)}-t^{(n)}$ and $\delta t_\disc=\max_{n=0,...,N-1}\delta t^{(n+\frac{1}{2})}$, to define
\begin{equation*}
\delta_\disc v(t)=\delta_\disc^{(n+\frac{1}{2})}\varphi:=\frac{\Pi_\disc(v^{(n+1)}-v^{(n)})}{\delta t^{(n+\frac{1}{2})}}.
\end{equation*}

\begin{definition}[Gradient scheme]\label{def-gs-obs} The gradient scheme for Problem \eqref{obs-lineaer} is to find sequences $u=(u^{(n)})_{n=0,...,N} \subset \cK_\disc$, such that $u^{(0)}=J_\disc u_{\rm ini}\in \cK_\disc$ and for all $n=0,...,N-1$,
\begin{equation}\label{gs-pvi-obs}
\begin{array}{ll}
\dsp\int_\O \delta_\disc^{(n+\frac{1}{2})}u(\x)\, \Pi_\disc(u^{(n+1)}(\x)-v(\x)) \ud \x\\
+\dsp\int_\O \Lambda(\x)\nabla_\disc u^{(n+1)}(\x)\cdot\nabla_\disc(u^{(n+1)}-v)(\x) \ud \x\\
\leq \frac{1}{\delta t^{(n+\frac{1}{2})}}\dsp\int_{t^{(n)}}^{t^{(n+1)}}\int_\O f(\x,t)\Pi_\disc(u^{(n+1)}-v)(\x) \ud \x \ud t,\; \mbox{for all } v\in \cK_\disc.
\end{array}
\end{equation}
\end{definition}

\begin{remark}[existence of solution]
Note that at any time step $(n+1)$, we need to solve a gradient scheme for a linear elliptic variational inequality: setting $\alpha=\frac{1}{\delta t^{(n+\frac{1}{2})}}$, find $u^{(n+1)} \in \cK_\disc$, such that for all $v\in \cK_\disc$,
\begin{equation}\label{ch5-new-ellptc}
b(u^{(n+1)},u^{(n+1)}-v) \leq L(u^{(n+1)}-v),
\end{equation}
with the bilinear form $b(v,w)$ and the linear form $L(w)$ respectively defined by 
\[
b(v,w)=\alpha \dsp\int_\O \Pi_\disc v  \Pi_\disc w \ud \x +\int_\O \nabla_\disc v \cdot \nabla_\disc w \ud \x,\; \mbox{ for all } v,w \in \cK_\disc \quad \mbox{ and }
\]
\[L(w)=\dsp\int_\O f \Pi_\disc w \ud \x + \alpha \int_\O \Pi_\disc u^{(n)} \Pi_\disc w \ud \x,\; \mbox{ for all } w \in \cK_\disc.\]
The assumptions for Stampacchia's theorem can easily be verified, and therefore there exists a unique weak solution to \eqref{ch5-new-ellptc}. This leads to the existence and uniqueness of the solution to \eqref{gs-pvi-obs}.
\end{remark}

We defined in \cite{AD14} three parameters to measure the quality of gradient schemes for elliptic variational inequalities. We still use these quantities to establish error estimates of the gradient schemes for PVIs. For the sake of completeness, we recall them as follows:

\begin{equation}\label{corc}
C_\disc =  {\dsp \max_{v \in X_{\disc,0}\setminus\{0\}}\frac{\|\Pi_\disc v\|_{L^2(\O)}}{\|\nabla_\disc v\|_{L^2(\O)^{d}}}},
\end{equation}
\begin{equation}\label{consist}
\begin{split}
&S_\disc : \cK \to [0, +\infty),\\
&\forall \varphi \in \cK, \quad S_\disc(\varphi)=\min_{v\in \cK_\disc\setminus\{0\}} \| \Pi_\disc v - \varphi \|_{L^2(\O)} 
+ \| \nabla_\disc v - \nabla \varphi \|_{L^2(\O)^{d}},
\end{split}
\end{equation}
\begin{equation}\label{conformityobs}
\begin{split}
&W_\disc : H_{\rm{div}}(\O)= \{\bo \in L^{2}(\O)^{d}{\;:\;} {\rm div}\bo \in L^2(\O)\} \to [0, +\infty)\\
&\forall \bo \in H_{\mathrm{div}}(\O),\\
&W_\disc(\bo)=
 \sup_{v\in X_{\disc,0}\setminus \{0\}}\frac{1}{\|\nabla_\disc v\|_{L^2(\O)^d}} \Big|\int_{\O}(\nabla_\disc v\cdot \bo + \Pi_\disc v \cdot\mathrm{div} (\bo)) \ud \x
\Big|.
\end{split}
\end{equation}

As mentioned previously, these parameters play an important role in obtaining error estimates and their corresponding rates. We apply the function $S_\disc$ and $W_\disc$ to the continuous solution and its gradient, respectively. \cite{S1} describes the relation between the functions $S_\disc$ and $W_\disc$ and the mesh size for mesh--based gradient discretisations for PDEs. The proof of such a link can be easily transferable to the above setting of gradient discretisations for variational inequalities, and gives 
\begin{subequations}\label{eq-h-disc}
\begin{align}
&S_\disc(\varphi)\leq h_\disc || \varphi ||_{H^2(\O)},\quad \forall \varphi\in H^2(\O)\cap \cK_\disc,\label{eq-s}\\
&W_\disc(\varphi)\leq h_\disc || \varphi ||_{H^1(\O)^d},
\quad \forall \varphi\in H^1(\O)^d,\label{eq-w}
\end{align}
\end{subequations}
where $h_\disc$ is the space size of the space time gradient discretisation defined by
\[
h_\disc=\max\Big( \dsp\sup_{\varphi \in (H^2(\O) \cap \cK)\setminus \{0\}}\frac{S_\disc(\varphi)}{|| \varphi ||_{H^2(\O)}} , 
\sup_{\varphi \in H^1(\O)^d\setminus \{0\}}\frac{W_\disc(\varphi)}{|| \varphi ||_{H^1(\O)^d}}
  \Big).
\]

\section{Main results}\label{sec-linear}
We present here our main error estimates for the gradient schemes approximations of our problem. In what follows, we denote by $I_\disc^0$ the error resulting from interpolation of the initial condition, which is given by 
\begin{equation}\label{eq-ID}
I_\disc^0=\Big\| u_{\ini}-\Pi_\disc J_\disc u_{\ini} \Big\|_{L^2(\O)}.
\end{equation}

With putting $\bar u^{(0)}=\bar u(0)$, we define the averaging over time in $(t^{(n)}, t^{(n+1)})$ as, for $n \in \{ 0,...,N-1 \}$, 
\begin{equation}\label{eq-g}
Z^{(n+1)}(\x)=\frac{1}{\delta t^{(n+\frac{1}{2})}}\int_{t^{(n)}}^{t^{(n+1)}}Z(\x,t) \ud t,\quad \mbox{ where $Z=f$, $\bar u$ or $\dr_t\bar u$.} 
\end{equation}

In general, obtaining error estimates hinges on finding a proper interpolant to plug the exact solution in the approximate scheme. In the case of parabolic problem, the interpolant must enjoy two properties; linearity and providing a better approximation. We assume here that there exists a linear continuous interpolant $P_\disc : \cK \to \cK_\disc$ and $C_P>0$ not depending on $\disc$, such that,  
\begin{equation}\label{eq-error-assumption}
|| \Pi_\disc P_\disc \varphi-\varphi ||_{L^2(\O)}^2
+|| \nabla_\disc P_\disc \varphi-\nabla\varphi ||_{L^2(\O)^d}^2
\leq C_P S_\disc(\varphi).
\end{equation}
We denote by $e_\disc^P$ the error corresponding to the interpolation of the exact solution, defined by 
\begin{equation}\label{eq-err-PD}
e_\disc^P=|| \bar u^{(n+1)}-\Pi_\disc\bar u(t^{(n+1)}) ||_{L^2(\O)}.
\end{equation}


\begin{theorem}[Error estimate]\label{theorem-error}
Let assumptions \ref{asm-2} hold and $\disc$ be a gradient discretisation. Let $u$ be the solution to the gradient scheme \eqref{gs-pvi-obs} and assume that the problem \eqref{obs-lineaer} has a solution $ \bar u \in W^{1,\infty}(0,T;H^2(\O))$. Furthermore, assume that there exists a linear continuous interpolant $P_\disc$ satisfying \eqref{eq-error-assumption}. Then there exists constants $A,\; B,\; C_F\geq 0$, depending only on $\bar u$, $\O$, $C_\disc$, $f$, $C_P$ and $T$ such that 
\begin{subequations}\label{eq-error}
\begin{align}
&\max_{t\in[0,1]}|| \Pi_\disc u(\cdot,t)-\bar u(\cdot,t) ||_{L^2(\O)}
\leq A(\delta t_\disc+h_\disc+I_\disc^0)+\left(\dsp\sum_{n=0}^{m-1}\delta t^{(n+\frac{1}{2})}C_F e_\disc^P \right)^{\frac{1}{2}},
\\
&|| \nabla_\disc u-\nabla\bar u ||_{L^2(\O\times(0,T))^d}
\leq B(\delta t_\disc+h_\disc+I_\disc^0)+\dsp\sum_{n=0}^{m-1}\delta t^{(n+\frac{1}{2})}C_F e_\disc^P,
\end{align}
\end{subequations}
where $I_\disc^0$ and $e_{P_\disc}$ are the errors due to the interpolation of the initial condition and the exact solutions defined by \eqref{eq-ID} and \eqref{eq-err-PD}, respectively.  
\end{theorem}

\begin{proof}
The proof is inspired from \cite{S1}. In this proof, the constants $C_i,\; i=1,2,...,9$ depend on $\O,\; \bar u,\; C_\disc,\; f$. Since $\nabla\bar u: [0,T] \to L^2(\O)^d$ 
is Lipschitz--continuous, we have, by applying \eqref{eq-error-assumption} and using \eqref{eq-s}, 
\begin{equation}\label{eq-proof1}
\begin{aligned}
\Big\| \nabla\bar u^{(n+1)}&-\nabla_\disc P_\disc \bar u(t^{(n+1)}) \Big\|_{L^2(\O)^d}\\
&\leq \Big\| \nabla\bar u^{(n+1)}-\nabla\bar u(t^{n+1}) \Big\|_{L^2(\O)^d} +S_\disc(\bar u(t^{n+1}))\leq C_1(\delta t_\disc+h_\disc).
\end{aligned}
\end{equation}
Since $|| \dr_t \bar u^{(n+1)} ||_{H^2(\O)}$ is bounded independently of $n$, we can use \eqref{eq-error-assumption} with $\varphi=\dr_t\bar u^{(n+1)}=\frac{\bar u(t^{(n+1)})-\bar u(t^{(n)})}{\delta t^{(n+\frac{1}{2})}}$. Then, by the linearity of $P_\disc$ and \eqref{eq-s}, we obtain
\begin{equation}\label{eq-proof2}
\Big\| \frac{\Pi_\disc P_\disc\bar u(t^{n+1})-\Pi_\disc P_\disc \bar u(t^{(n)})}
{\delta t^{n+\frac{1}{2}}}
-\partial_t \bar u^{(n+1)} \Big\|_{L^2(\O)}
\leq C_2h_\disc.
\end{equation}
Note that $\nabla\bar u^{(n+1)} \in H_{\div}$. Inequality \eqref{conformityobs} with $\bo=\nabla\bar u^{(n+1)}$ gives
\begin{equation}\label{eq-limi-conf}
\begin{aligned}
\forall w\in X_{\disc,0},\; \int_\O \Big(\Pi_\disc w(\x)\div(\nabla\bar u^{(n+1)})(\x)
&+\nabla\bar u^{(n+1)}\cdot \nabla_\disc w(\x) \Big) \ud \x \\
&\leq C_3h_\disc || \nabla_\disc w ||_{L^2(\O)^d}.
\end{aligned}
\end{equation}
Regularity assumptions on the solution $\bar u$ show that \eqref{pvi-obs1} holds a.e. in space and time. Averaging over time in $(t^{(n)},t^{(n+1)})$ leads to $\dr_t \bar u^{(n+1)}-f^{(n+1)} \leq \div(\nabla\bar u^{(n+1)})$. Since $u \in \cK_\disc$, we get $\int_\O (\Pi_\disc u^{(n+1)}-\psi)(f^{(n+1)}-\div(\nabla\bar u^{(n+1)})-\dr_t\bar u^{(n+1)}) \ud x \leq 0$. Hence, for any $v\in \cK_\disc$, we can write
\begin{align*}
\int_\O \Pi_\disc(u^{(n+1)}(\x)&-v(\x))\div(\nabla\bar u^{(n+1)}(\x)) \ud \x\\
&\leq\int_\O (\psi(\x)-\Pi_\disc v(\x))(f^{(n+1)}(\x)-\div(\nabla\bar u^{(n+1)}(\x))-\dr_t\bar u^{(n+1)}(\x))\\
&-\int_\O \Pi_\disc(u^{(n+1)}(\x)-v(\x))(f^{(n+1)}(\x)-\dr_t\bar u^{(n+1)}(\x)) \ud \x,
\end{align*}
which gives with introducing $\bar u^{(n+1)}$ in the first term
\begin{align*}
\int_\O &\Pi_\disc(u^{(n+1)}(\x)-v(\x))\div(\nabla\bar u^{(n+1)}(\x)) \ud \x\\
&\leq \int_\O (\psi(\x)-\bar u^{(n+1)}(\x))(f^{(n+1)}(\x)-\div(\nabla\bar u^{(n+1)}(\x))-\dr_t\bar u^{(n+1)}(\x)) \ud \x
\\
&+\int_\O (\bar u^{(n+1)}(\x)-\Pi_\disc v(\x))(f^{(n+1)}(\x)-\div(\nabla\bar u^{(n+1)}(\x))-\dr_t\bar u^{(n+1)}(\x)) \ud \x\\
&-\int_\O \Pi_\disc(u^{(n+1)}(\x)-v(\x))(f^{(n+1)}(\x)-\dr_t\bar u^{(n+1)}(\x)) \ud \x.
\end{align*}
Since the first term on the R.H.S is zero, this inequality can be written as
\begin{align*}
\int_\O \Pi_\disc(v(\x)&-u^{(n+1)}(\x))\div(\nabla\bar u^{(n+1)}(\x)) \ud \x\\
&\geq
\int_\O \Pi_\disc(v(\x)-u^{(n+1)}(\x))(\dr_t\bar u^{(n+1)}(\x)-f^{(n+1)}(\x)) \ud \x\\
&+\dsp\int_\O(\bar u^{(n+1)}(\x)-\Pi_\disc v(\x))(f^{(n+1)}(\x)-\div(\nabla\bar u^{(n+1)})(\x)-\dr_t\bar u^{(n+1)}(\x)) \ud \x.
\end{align*}
Consider $w=v-u^{(n+1)}\in X_{\disc,0}$ in \eqref{eq-limi-conf}. Employ the above inequality to replace $\div(\nabla\bar u^{(n+1)})$ in the left-hand side to find
\begin{align*}
\int_\O \Pi_\disc(v(\x)&-u^{(n+1)}(\x))(\dr_t\bar u^{(n+1)}(\x)-f^{(n+1)}(\x)) \ud \x\\
&+\int_\O \nabla_\disc(v(\x)-u^{(n+1)}(\x))\cdot \nabla\bar u^{(n+1)}(\x)) \ud \x\\
&\leq C_3h_\disc || \nabla_\disc (v-u^{(n+1)}) ||_{L^2(\O)^d}\\
&+\dsp\int_\O(\bar u^{(n+1)}(\x)-\Pi_\disc v(\x))(f^{(n+1)}(\x)-\div(\nabla\bar u^{(n+1)})(\x)-\dr_t\bar u^{(n+1)}(\x)) \ud \x.
\end{align*}
Therefore, we use the fact that $u$ is the solution to the gradient scheme \eqref{gs-pvi-obs} to arrive at
\begin{equation}\label{eq-proof3}
\begin{aligned}
\int_\O \Pi_\disc(v(\x)&-u^{(n+1)}(\x))(\partial_t\bar u^{(n+1)}(\x)-\delta_\disc^{(n+\frac{1}{2})}u(\x)) \ud \x\\
&+\int_\O \nabla_\disc (v(\x)-u^{(n+1)}(\x))\cdot (\nabla\bar u^{(n+1)}(\x)-\nabla_\disc u^{(n+1)}(\x)) \ud \x\\
&\leq C_3h_\disc || \nabla_\disc (v-u^{(n+1)}) ||_{L^2(\O)^d}\\
&+\dsp\int_\O(\bar u^{(n+1)}(\x)-\Pi_\disc v(\x))(f^{(n+1)}(\x)-\div(\nabla\bar u^{(n+1)})(\x)-\dr_t\bar u^{(n+1)}(\x)) \ud \x.
\end{aligned}
\end{equation}
For $k=1,...,N$, denote $e^{(k)}:=P_\disc\bar u(t^{(k)})-u^{(k)}$ to introduce
\[
\delta_\disc^{(n+\frac{1}{2})}e=\Big(  \dsp\frac{ \Pi_\disc P_\disc\bar u(t^{(n+1)})-\nabla\bar u^{(n+1)} }{ \delta t^{(n+\frac{1}{2})} }-\dr_t\bar u^{(n+1)} \Big)+\Big( \dr_t\bar u^{(n+1)}-\delta_\disc^{(n+\frac{1}{2})} u \Big),
\]
and
\[
\nabla_\disc e^{(n+1)}= \Big( \nabla_\disc P_\disc\bar u(t^{(n+1)})-\nabla\bar u^{(n+1)}  \Big) + \Big(\nabla\bar u^{(n+1)}+\nabla_\disc u^{(n+1)}  \Big).
\]
Therefore, from \eqref{eq-proof1}, \eqref{eq-proof2} and \eqref{eq-proof3}, and the upper bound of $C_\disc$, we attain
\begin{align*}
\int_\O \Pi_\disc(v(\x)&-u^{(n+1)}(\x))\delta_\disc^{(n+\frac{1}{2})}e(\x) \ud \x
+\int_\O \nabla_\disc(v(\x)-u^{(n+1)}(\x))\cdot \nabla_\disc e^{(n+1)}(\x) \ud \x\\
&\leq\dsp\int_\O(\bar u^{(n+1)}(\x)-\Pi_\disc v(\x))(f^{(n+1)}(\x)-\div(\nabla\bar u^{(n+1)})(\x)-\dr_t\bar u^{(n+1)}(\x)) \ud \x\\
&+C_3(\delta_\disc+h_\disc) || \nabla_\disc (v-u) ||_{L^2(\O)^d}.
\end{align*}
Apply this estimate to $v=P_\disc \bar u(t^{(n+1)})$, multiply by $\delta t^{(n+\frac{1}{2})}$ and sum over $n=0,...,m-1$ for some $m\in \{ 1,...,N \}$ to deduce
\begin{align*}
\dsp\sum_{n=0}^{m-1} \int_\O \Pi_\disc e^{(n+1)}(\x)\Big[\Pi_\disc e^{(n+1)}(\x)-\Pi_\disc e^{(n)}(\x)\Big] \ud \x
+\dsp\sum_{n=0}^{m-1} \delta t^{(n+\frac{1}{2})} \Big\| \nabla_\disc e^{(n+1)} \Big\|_{L^2(\O)^d}^2\\
\leq \dsp\sum_{n=0}^{m-1} \delta t^{(n+\frac{1}{2})} C_4(\delta t_\disc+h_\disc) \Big\| \nabla_\disc e^{(n+1)} \Big\|_{L^2(\O)^d}
+\dsp\sum_{n=0}^{m-1}\delta t^{(n+\frac{1}{2})}E_\disc^{(n+1)},
\end{align*}
where
\begin{equation}\label{eq-ED}
\begin{aligned}
E_\disc^{(n+1)}:=\dsp\int_\O(\bar u^{(n+1)}(\x)&-P_\disc \bar u(t^{(n+1)}))(f^{(n+1)}(\x)\\
&-\div(\nabla\bar u^{(n+1)})(\x)-\dr_t\bar u^{(n+1)}(\x)) \ud \x.
\end{aligned}
\end{equation}
Using the relation $b(a-b)\geq \frac{1}{2}b^2-\frac{1}{2}a^2$ with $a=\Pi_\disc e^{(n)}(\x)$ and $b=\Pi_\disc e^{(n+1)}(\x)$ and Young's inequality, the above estimate yields (note that $\sum_{n=0}^{m-1}\delta t^{(n+\frac{1}{2})} \leq T$) 
\begin{equation}\label{eq-proof4}
\begin{aligned}
\dsp\int_\O \frac{1}{2}(\Pi_\disc e^{(m)}(\x))^2 \ud \x
&+\dsp\sum_{n=0}^{m-1} \delta t^{(n+\frac{1}{2})} \Big\| \nabla_\disc e^{(n+1)} \Big\|_{L^2(\O)^d}^2 \\
&\leq \frac{1}{2}\dsp\int_\O (\Pi_\disc e^{(0)}(\x))^2 \ud \x
+C_5(\delta t_\disc+h_\disc)^2\\
&+\frac{1}{2}\dsp\sum_{n=0}^{m-1} \delta t^{(n+\frac{1}{2})}\Big\| \nabla_\disc e^{(n+1)} \Big\|_{L^2(\O)^d}^2\\
&+\dsp\sum_{n=0}^{m-1}\delta t^{(n+\frac{1}{2})}E_\disc^{(n+1)}.
\end{aligned}
\end{equation}
Since $u^{(0)}= J_\disc u_{\ini}=J_\disc\bar u(0)$, from \eqref{eq-s} and \eqref{eq-error-assumption}, we get
\begin{align*}
|| \Pi_\disc e^{(0)} ||_{L^2(\O)} &\leq || \Pi_\disc P_\disc \bar u(0)-\bar u(0) ||_{L^2(\O)}
+|| \bar u(0)-\Pi_\disc J_\disc\bar u(0) ||_{L^2(\O)}\nonumber\\
&\leq C_4h_\disc + I_\disc^0.
\end{align*}
From the definitions \eqref{eq-err-PD} and \eqref{eq-ED}, we can see that there exists $C_F$ not depending on $\disc$, such that, $E_\disc^{(n+1)} \leq C_F e_\disc^P$. Substitute this estimate in Equation \eqref{eq-proof4} to obtain 
\begin{equation}\label{eq-proof5}
\begin{aligned}
\frac{1}{2} \Big|| \Pi_\disc e^{(m)} \Big||_{L^2(\O)}^2
&+\frac{1}{2}\dsp\sum_{n=0}^{m-1} \delta t^{(n+\frac{1}{2})} \Big\| \nabla_\disc e^{(n+1)} \Big\|_{L^2(\O)^d}^2\\
&\leq C_6(\delta t_\disc+h_\disc + I_\disc^0)^2
+\dsp\sum_{n=0}^{m-1}\delta t^{(n+\frac{1}{2})}C_F e_\disc^P.
\end{aligned}
\end{equation}
Introduce $\Pi_\disc P_\disc \bar u(t^{(m)})$, use a triangle inequality and \eqref{eq-error-assumption} and \eqref{eq-proof5} to get, for all $m=1,...,N-1$,
\begin{equation}\label{eq-proof6}
\begin{aligned}
\Big\| \Pi_\disc u^{(m)}-\bar u(t^{(m)}) \Big\|_{L^2(\O)}^2
&\leq C_6(\delta t_\disc+h_\disc + I_\disc^0)
+S_\disc(\bar u(t^{(m)}))\\
&+\dsp\sum_{n=0}^{m-1}\delta t^{(n+\frac{1}{2})}C_F e_\disc^P\\
&\leq
C_7(\delta t_\disc+h_\disc+I_\disc^0)+\left({\dsp\sum_{n=0}^{m-1}\delta t^{(n+\frac{1}{2})}C_F e_\disc^P}\right)^{\frac{1}{2}}.
\end{aligned}
\end{equation}
Similarly, introduce $\nabla_\disc P_\disc \bar u(t^{(m)})$, use a triangle inequality and \eqref{eq-error-assumption} and \eqref{eq-proof5} with $m=N-1$ to get the following estimation
\begin{equation}\label{eq-proof7}
\begin{aligned}
\dsp\sum_{n=0}^{N-1}\delta t^{(n+\frac{1}{2})} \Big\| \nabla_\disc u^{(n+1)}&-\nabla\bar u(t^{(n+1)}) \Big\|_{L^2(\O)^d}^2\\
&\leq 4C_8(\delta t_\disc+h_\disc+I_\disc^0)^2
+4\dsp\sum_{n=0}^{N-1}\delta t^{(n+\frac{1}{2})}S_\disc(\bar u(t^{(n+1)}))^2\\
&+\dsp\sum_{n=0}^{m-1}\delta t^{(n+\frac{1}{2})}C_F e_\disc^P\\
&\leq C_9^2 (\delta t_\disc+h_\disc+I_\disc^0)^2+\dsp\sum_{n=0}^{m-1}\delta t^{(n+\frac{1}{2})}C_F e_\disc^P.
\end{aligned}
\end{equation}
The desired estimates are implied by combing \eqref{eq-proof6} and \eqref{eq-proof7}, together with the Lipschitz--continuity of $\bar u : [0,T] \to H^1(\O)$, which is used to estimate the quantities $\bar u(t)-\bar u(t^{(n+1)})$ and $\nabla\bar u(t)-\nabla\bar u(t^{(n+1)})$ when $t \in (t^{(n)}, t^{(n+1)}]$. 
\end{proof}

\begin{remark}
The assumption of existence of the interpolant $P_\disc$ can always be satisfied. For PDEs problems, it is shown that there is an explicit linear interpolant $P_\disc$ such that \eqref{eq-error-assumption} holds, see \cite[Step 1 in the proof of Theorem 5.3]{S1}. While this interpolant is no longer valid to preserve the bound by the obstacle $\psi$ on a particular method, it is possible for specific methods to construct more appropriate $P_\disc$. For instance, the interpolant introduced in the appendix of \cite{YD-2018} will work in the case of hybrid mimetic mixed method.
\end{remark}

\begin{remark}
Theorem \ref{theorem-error} provides order of convergence in terms of $h_\disc$ and $\delta t_\disc$. It is explained in \cite{S1} that there exists a constant $C$ depending only on the regularity of mesh such that $h_\disc \leq Ch_\mesh^\alpha$, where $h_\mesh$ is the mesh size and $\alpha$ is the highest degree of the polynomials used to approximate the solution. The error estimates given in the theorem seem to be dominated by the term including $e_\disc^P$, which depends on the choice of the interpolant. Initially, this term seems to behave as $\sqrt h_\mesh$ for the first order conforming and non conforming methods. However, as we show in \cite{AD14}, this term can lead to the expected $\mathcal O(h_\mesh)$ convergence rate for the first order conforming and non conforming methods.
\end{remark}

\begin{remark}
In most numerical schemes, the constructed interpolant of smooth functions might not satisfy the obstacle condition $\psi$ inside the domain, especially if $\psi$ is not constant. It is classical to  consider only approximate obstacle $\psi_\disc \in L^2(\O)$ in the schemes to define the convex set
\begin{equation}\label{set-K}
\cK_{\disc,\psi_\disc}:=\{ v \in X_{\disc,0}: \Pi_\disc \leq \psi_\disc \}.
\end{equation}  
The scheme \eqref{gs-pvi-obs} is therefore modified by replacing the set $\cK_\disc$ by the set $\cK_{\disc,\psi_\disc}$. The error estimate for this case of approximate obstacle is presented in the following theorem, which can exactly be proved as the previous one (Theorem \ref{theorem-error}), see \cite[Section 6]{AD14} for dealing with the approximate obstacle $\psi$.
\end{remark}

\begin{theorem}
Under the assumptions of Theorem \ref{theorem-error}, if $\cK_{\disc,\psi_\disc}$ is not empty then there exists a unique solution $u$ to the gradient scheme \eqref{gs-pvi-obs} in which $\cK_\disc$ has been replaced
with $\cK_{\disc,\psi_\disc}$. Moreover, Estimates \eqref{eq-error} for the approximate obstacle still hold, provided that, $e_\disc^P$ is replaced with
\[
\tilde e_\disc^P:=e_\disc^P+|| \psi^{(n+1)}-\psi_\disc ) ||_{L^2(\O)}.
\]
\end{theorem}

\section{Numerical results}\label{test}
In this section, we demonstrate the efficiency of Scheme \eqref{gs-pvi-obs} with a particular choice of the gradient discretisation, corresponding to Hybrid Mimetic Mixed (HMM) method. It is a common framework  gathering three different methods: the hybrid finite volume method \cite{D-2010-SUSHI}, the (mixed--hybrid) mimetic finite differences methods \cite{Brezzi-1991}, and the mixed finite volume methods \cite{T1997}. For the sake of completeness we briefly recall the definition of this gradient discretisation. Let $\mathcal T=(\mesh,\edges,\mathcal P)$ be the polytopal mesh of the spatial domain $\O$, where $\mesh$ is the set of polygonal cells $K$, $\edges$ is the set of edges $\edge$, and $\mathcal P$ is a set of points $(x_K)_{K\in\mesh}$. The elements of GD are: 
\begin{enumerate}
\item The discrete space and set are 
\[
X_{\disc,0}=\{ v=((\varphi_{K})_{K\in \mathcal{M}}, (\varphi_{\sigma})_{\sigma \in \mathcal{E}})\;:\; \varphi_{K},\, \varphi_{\sigma} \in \RR\; \varphi_\edge=0,\; \forall \edge \in \edges \cap \dr\O
\},
\]
\[
\cK_\disc=\{ v=((\varphi_{K})_{K\in \mathcal{M}}, (\varphi_{\sigma})_{\sigma \in \mathcal{E}}) \in X_{\disc,0}\;:\; \varphi_K \geq \psi,\; \forall K\in \mesh
\},
\]
where the $x_\edge$ is centre of mass of $\edge$.
\item The non conforming a piecewise affine reconstruction $\Pi_\disc$ is defined by
\[
\begin{aligned}
&\forall \varphi\in X_{\disc,0}, \forall K\in \mesh, \mbox{ for a.e. } \x \in K,\\
&\Pi_\disc \varphi=\varphi_K\mbox{ on $K$},
\end{aligned}
\]
\item The reconstructed gradients is piecewise constant on the cells (broken gradient), defined by
\[
\begin{aligned}
&\forall \varphi\in X_{\disc,0},\; \forall K\in\mathcal M,\,\forall \sigma\in\mathcal E_K,\\
&\nabla_\disc \varphi=\nabla_{K}\varphi+
\frac{\sqrt{d}}{d_{K,\sigma}}R_K(\varphi)\mathbf{n}_{K,\sigma} \mbox{ on } D_{K,\edge},
\end{aligned}
\]
where a cell--wise constant gradient $\nabla_K(\varphi)$ and a stabilisation term $R_K(\varphi)$ are respectively defined by:
\[
\nabla_{K}\varphi= \dsp\frac{1}{|K|}\sum_{\sigma\in \edgescv}|\sigma|\varphi_\edge\mathbf{n}_{K,\sigma} \mbox{ and } R_K(\varphi)=(\varphi_\edge - \varphi_K - \nabla_K \varphi\cdot(\centeredge-x_K))_{\edge\in\edges_K}.
\]
in which $d_{K,\edge}$ is the orthogonal distance between $x_K$ and $\edge \in \edges_K$, ${\bf n}_{K,\edge}$ is the unit vector normal to $\edge$ outward to $K$ and $D_{K,\edge}$ is the convex hull of $\edge \cup \{x_K\}$.
\item The interpolant $J_\disc: L^2(\O) \to X_{\disc,0}$ is defined by:
\[
\begin{aligned}
&\forall w \in L^2(\O)\;:\; J_\disc w=((w_K)_{K\in \mesh},(w_\edge)_{\edge\in \edges}),\\
&\forall K\in\mesh,\; w_K=\dsp\frac{1}{|K|}\dsp\int_K w(\x) \ud \x \mbox{ and } \forall \edge\in\edges,\; w_\edge=0.
\end{aligned}
\]
\end{enumerate}

The HMM scheme for \eqref{obs-lineaer} is the gradient scheme \eqref{gs-pvi-obs} written with the above constructed GD. For computation purpose, we will transfer the HMM method into the finite volume formats. Let us define the linear fluxes $u\mapsto F_{K,\sigma}(u)$ (for $K\in\mesh$ and $\sigma\in\edges_K$) by, for all $K \in \mesh$ and all $u,v\in X_{\disc,0}$,
\begin{align*}
\sum_{\sigma \in \mathcal{E}_K}|\sigma| F_{K,\sigma}(u)
(v_K-v_\sigma)={}&\int_K \nabla_\disc u\cdot\nabla_\disc v \ud \x.
\end{align*} 

The HMM method for Model \eqref{pvi-obs} is, for all $K \in \mesh$ and for all $n=0,...,N-1$, the following holds
\begin{subequations}\label{mont-obs}
\begin{equation}
\left(\frac{|K|}{\delta t^{(n+\frac{1}{2})}}\left( u^{(n+1)}-u^{(n)} \right)+\sum_{\sigma \in \mathcal{E}_K} F_{K,\sigma}(u^{(n+1)})- |K|f_K^{(n+1)}\right)(u_K^{(n+1)}-\psi_K)=0,
\end{equation}
\begin{equation}
\frac{|K|}{\delta t^{(n+\frac{1}{2})}}\left( u^{(n+1)}-u^{(n)} \right)+\sum_{\sigma \in \mathcal{E}_K} F_{K,\sigma}(u^{(n)})\geq |K|f_K^{(n+1)},
\end{equation}
\begin{equation}
u_K^{(n+1)} \geq \psi_K,
\end{equation}
\begin{equation}\label{obs-hmm-d}
F_{K,\edge}(u^{(n+1)})+F_{L,\edge}(u^{(n+1)})=0, \quad \forall \sigma \in \edges_K \cap \edges_L, K\neq L,
\end{equation}
\begin{equation}\label{obs-hmm-e}
u_\sigma^{(n+1)}=0, \quad \forall \sigma \in \mathcal E_{\rm ext}.
\end{equation}
\end{subequations}

Solving this non linear model can be expensive. At each time step $t^{(n)}$, a system of inequalities must be solved. We use the monotonicity iterations Algorithm \ref{algo:monoton} detailed in \cite{A-2} to solve this system; the non linearity caused by the inequalities in the model is eliminated and thus we deal, at each of its steps, with a square linear system on unknowns on each iteration. The number of cells is the upper bound of the number of iterations.
\begin{algorithm}
\caption{Monotonicity algorithm}\label{algo:monoton}
\begin{algorithmic}[1]
\State (Only the first time the algorithm is called):

\hspace*{-2em}\begin{minipage}{0.4\linewidth}
Set $\cA^{(0)}=\mesh$, $\cB=\emptyset$\\
and $I={\rm Card}(\cA^{(0)})$
\end{minipage} \Comment{$I$= theoretical bound on the iterations}
\State $u^{(n)}$ being known and $u^{(n+1)}$ is the solution to \eqref{mont-obs} at time step $t^{(n+1)}$
\While{$i\le I$} 
\State $\cA^{(i)}$ and $\cB^{(i)}$ being known, find the solution $u^{(i)}$ to \eqref{obs-hmm-d} and \eqref{obs-hmm-e} with
\begin{subequations}
\begin{align*}
\frac{|K|}{\delta t^{(n+\frac{1}{2})}} u^{(i)}+\sum_{\sigma \in \mathcal{E}_K}F_{K,\sigma}(u^{(i)})=|K|f_K+\frac{|K|}{\delta t^{(n+\frac{1}{2})}}u^{(n)}, &\quad \forall K \in \cA^{(i)}\\
u_K^{(i)} = \psi_K^{(n+1)},&\quad \forall K \in \cB^{(i)}.  
\end{align*}
\end{subequations}
\State Set $\cA^{(i+1)}$=:
\begin{align*}
\{ K\in \cA^{(i)}\; :\; &\frac{|K|}{\delta t^{(n+\frac{1}{2})}} u^{(i)} +\sum_{\sigma \in \edges} F_{K,\sigma}(u^{(i)})
\geq |K|f_K^{(n+1)} + \frac{|K|}{\delta t^{(n+\frac{1}{2})}} u^{(n)} \} \\
&\cup \{ K\in \cB^{(i)}\; :\; u_K^{(i)}\leq \psi_K \}
\end{align*}
\State Set $\cB^{(i+1)}=:$
\begin{align*}
\{ K\in \cB^{(i)}\; :\; &\frac{|K|}{\delta t^{(n+\frac{1}{2})}} u^{(i)} +\sum_{\sigma \in \edges} F_{K,\sigma}(u^{(i)})
\leq |K|f_K^{(n+1)} + \frac{|K|}{\delta t^{(n+\frac{1}{2})}} u^{(n)} \}\\ 
&\cup \{ K\in \cA^{(i)}\; :\; u_K^{(i)}\geq \psi_K \}
\end{align*}
\If {$\cA^{(i+1)}=\cA^{(i)}$ and $\cB^{(i+1)}=\cB^{(i)}$} 
\State {Exit ``while'' loop}
\EndIf
\EndWhile
\State Set $u^{(n+1)}=u^{(i)}$ \Comment{Solution to \eqref{mont-obs} at time step $t^{(n+1)}$}
\State (For next call of Algorithm \ref{algo:monoton}) Set $\cA^{(0)}=\cA^{(i+1)}$ and $\cB^{(0)}=\cB^{(i+1)}$
\end{algorithmic}
\end{algorithm}

Now, we perform two different numerical tests taken from the literature to measure the validity of the error estimates proved in Theorem \ref{theorem-error}. Unlike previous experiments, we conduct our tests on two families of triangular meshes and hexagonal meshes (as in Figure \ref{Fig.test-1-mesh}). In both tests, we consider the model \eqref{pvi-obs}, in which the spatial domain $\O=(-1,1)^2$ and $\Lambda={\bf Id}$. 

\begin{test}\label{test-1}
We consider a test with an analytical solution introduced in \cite{Moon-2007}. Let the final time $T=0.25$, and $\psi \equiv 0$. The non-contact and contact sets are $\O^+:=\{ \x\in \O\; : \; r(t) > s(t) \} \mbox{ and } \O^{-}=\O-\O^+ 
$. In this test, the functions $r,\; s,\; q_1,\; q_2:[0,T] \to \RR^+$ are given by
\[
r(t)=\Big( (x-\frac{1}{3}\cos(4\pi t))^2 + (y-\frac{1}{3}\sin(4\pi t))^2\Big)^{\frac{1}{2}}, \quad s(t)=\frac{1}{3}+0.3 \sin(16\pi t).
\]
\[
q_1(t)=\frac{1}{3}\cos(4\pi t), \quad \mbox{ and } q_2(t)=\frac{1}{3}\sin(4\pi t).
\]
The source term function is
\[
f(\x,t)=\begin{cases}
\hfill 4 \Big( s^2(t)-2r^2(t)-\frac{1}{2}( r^2(t)-s^2(t) ) \left(p(t)+s(t)\dr_t s \right) \Big)&, \; \mbox{ if }\; \x \in \O^+,\\
-4 s^2(t)\left( 1-r^2(t)+ s^2(t) \right)&, \; \mbox{ if }\; \x \in \O^-.
\end{cases}
\]
where $p(t)=(x-q_1(t))q_1^{'}(t) + (y-q_2(t))q_2^{'}(t)$.

The initial and boundary conditions are imposed by the solution $\bar u$, defined by
\[
\bar u(\x,t)=\begin{cases}
\hfill \frac{1}{2} \Big( r^2(t) - s^2(t) \Big)^2&, \; \mbox{ if }\; \x \in \O^+,\\
0&, \; \mbox{ if }\; \x \in \O^0.
\end{cases}
\]

We present in Figure \ref{Fig-test-1-sol} the surface plots of the approximate solutions for the third mesh in each family at the final time $T=0.25$. Table \ref{tab-test-1} details the relative errors on $\bar u$ and $\nabla\bar u$ and the corresponding convergence rate with respect to the mesh size. We use the explicit Euler scheme with a uniform time step $\delta t^{(n+\frac{1}{2})}= \mathcal O(h^2)$ to obtain the results. The observed numerical rates with respect to the mesh size is $1$, which matches the expectation of Theorem \ref{theorem-error} for the HMM method.


\medskip
\begin{figure}[ht]
	\begin{center}
	\includegraphics[scale=0.70]{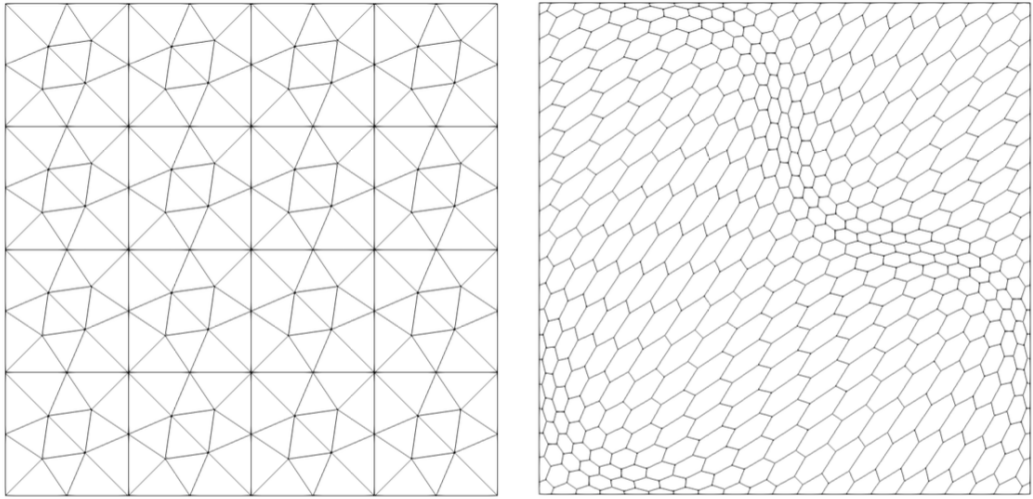}
	\end{center}
	\caption{First two elements in each type of mesh family.}
	\label{Fig.test-1-mesh}
\end{figure}

\medskip
\begin{figure}[ht]
	\begin{center}
	\includegraphics[scale=0.70]{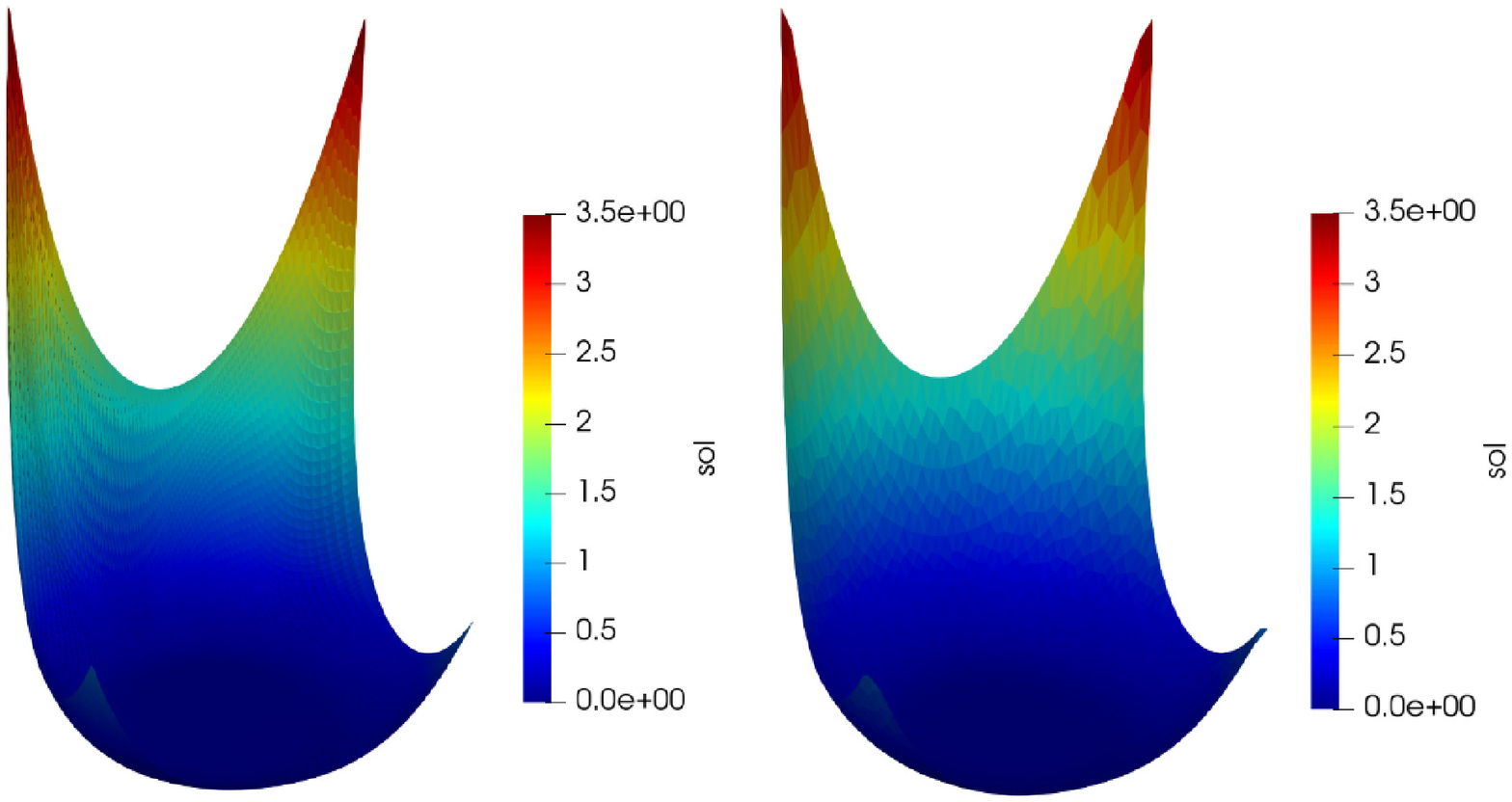}
	\end{center}
	\caption{Test \ref{test-1}. Surface plot of the solution on a hexahedral mesh (left) and on a triangular mesh (right) at final time ($T=0.25$).}
	\label{Fig-test-1-sol}
\end{figure}

\begin{table}[]
\begin{tabular}{c c c c c}
\hline\\
$h$&
$\frac{\| \bar u(\cdot,T) - \Pi_\disc u^N\|_{L^{2}(\O)}}{\|\bar u(\cdot,T)\|_{L^2(\O)}}$&
rate&
$\frac{\| \nabla \bar u(\cdot,T) - \nabla_\disc u^N\|_{L^{2}(\O)^2}}{\|\nabla\bar u(\cdot,T)\|_{L^2(\O)^2}}$&
rate  
\\ \hline
\multicolumn{5}{c}{a hexagonal mesh} \\ 
0.48&
0.14586&
--&
0.22031&
--
\\ 
0.26&
0.06532&
1.29&
0.12743&
0.88
\\ 
0.13&
0.00939&
2.85&
0.06678&
0.95
\\ 
0.07&
0.00440&
1.10&
0.06108&
0.13
\\ \hline
\multicolumn{5}{c}{a triangular mesh} \\ 
 0.25 & 0.13997  & --  & 0.25249  & --  \\ 
 0.18 & 0.05639  & 2.62  & 0.13138  & 1.88  \\ 
 0.09 & 0.01116  & 2.34  & 0.06942  & 0.92  \\ 
 0.03 & 0.00447  & 0.88  & 0.06161  & 0.11 \\
 \hline 
\end{tabular}
\caption{Test \ref{test-1}: relative errors and and convergence rates w.r.t. the mesh size $h$, for uniform time steps $\delta t^{(n+\frac{1}{2})}=h^2$.}
\label{tab-test-1}   
\end{table}
\end{test}

\begin{test}\label{test-2}
We consider our model with particular data as in \cite{PVI-test-2009}; $T=0.1$, $f=-4$ and the barrier function $\psi$ is
\[
\psi(x,y)=\max\{0,-0.1+0.6 \exp(-10 r^2), 0.5 -r\},\quad \mbox{with}\quad r=\sqrt{x^2+y^2}.
\]

Figures \ref{Fig-test-2-sol} and \ref{Fig-test-2-difference} present the HMM solution to the above obstacle problem and the difference function $u-\psi$ computed at the final time step $t=0.1$, respectively

Figure \ref{Fig-test-2-coincidence} displays the coincidence set based. The black area presents the set of cell centers where the approximate solution $u$ reaches the barrier $\psi$. The contact regions are very similar to the ones obtained by the finite difference method in \cite{PVI-test-2009}. For instance, the maximum $y$ ordinate of points $\x \in \O$, where the solution is strictly larger than the obstacle, is located around $y=0.6$.

\par Solving parabolic variational inequalities in practice is more expensive than linear parabolic partial differential equations models. At each time step, we iterate to solve a number of systems of elliptic equations (see Algorithm \ref{algo:monoton}). To determine the initial two sets $\cA$ and $\cB$ introduced in Algorithm \ref{algo:monoton}, we assume that $\cA^{(0)}=\mesh$, that is the solution is everywhere equal to the barrier at initial step. After determining the final $\cA^{(N)}$ and $\cB^{(N)}$ at time $t^{(n)}$, we use these set as initial guess for the monotonicity algorithm at time $t^{(n+1)}$. Given that the solution to the PVI is not expected to move a lot between $t^{(n)}$ and $t^{(n+1)}$, these initial guesses are not far from the correct regions at time $t^{(n+1)}$. As a consequence, the number of iterations is reduced as the time step increases: from $11$ iteration (starting from the guess $I^{(0)}=\mesh$) at $t^{(1)}$ to $2$ iterations at time $t^{(4)}$ and after. The iterations number of the algorithm required to reach the solution at any time step is ranged from 2 to 11 iterations.

\medskip
\begin{figure}[ht]
	\begin{center}
	\includegraphics[scale=0.70]{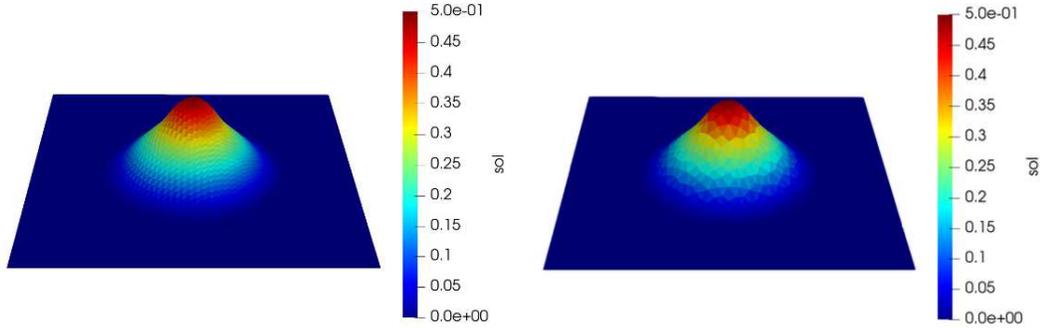}
	\end{center}
	\caption{Test \ref{test-2}. Surface plot of the solution on a hexahedral mesh (left) and on a triangular mesh (right) at final time ($T=0.25$).}
	\label{Fig-test-2-sol}
\end{figure}

\medskip
\begin{figure}[ht]
	\begin{center}
	\includegraphics[scale=0.70]{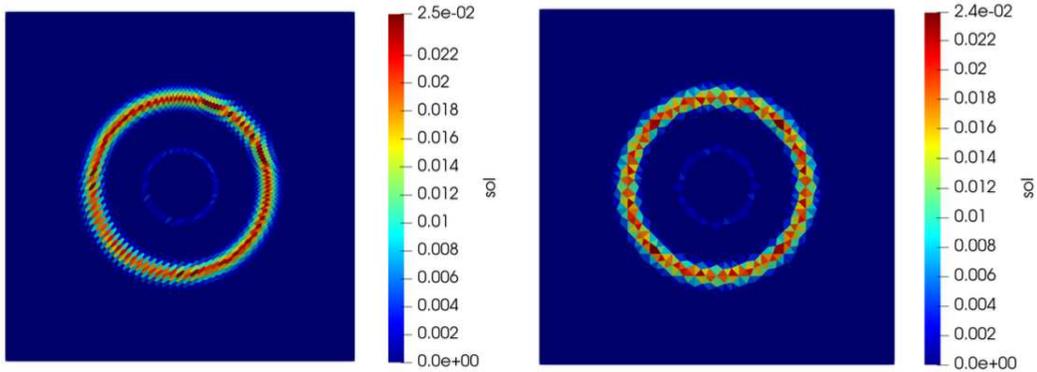}
	\end{center}
	\caption{Test \ref{test-2}. Surface plot of the difference function ($u-\psi$) on a hexahedral mesh (left) and on a triangular mesh (right) at final time ($T=0.25$).}
	\label{Fig-test-2-difference}
\end{figure}

\medskip
\begin{figure}[ht]
	\begin{center}
	\includegraphics[scale=0.70]{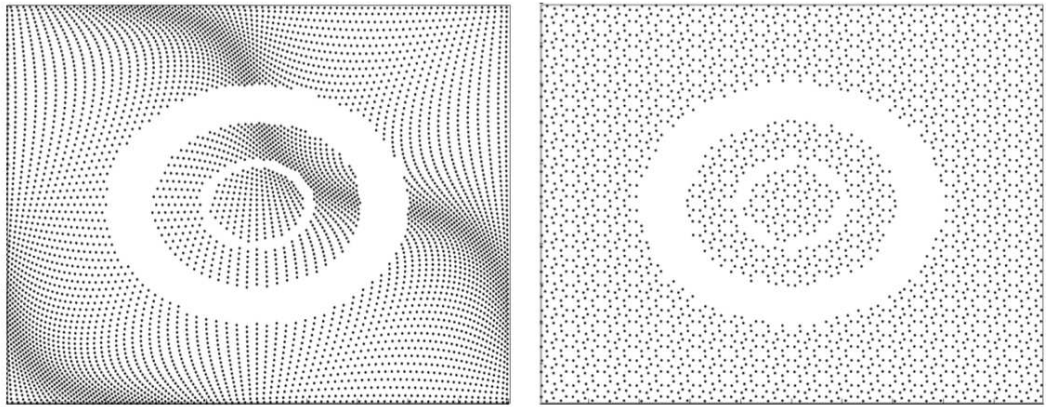}
	\end{center}
	\caption{Test \ref{test-2}. Plot of the coincidence set on a hexahedral mesh (left) and a triangular mesh (right) at final time ($T=0.25$).}
	\label{Fig-test-2-coincidence}
\end{figure}

\end{test}

Figures \ref{Fig-test-1-K-mesh} and \ref{Fig-test-2-K-mesh} represent our experiments obtained when using a "Kershaw" mesh as in the FVCA5 benchmark \cite{HH08}. As expected on these kinds of extremely distorted meshes, the results in both tests seem to be quantitively and qualitatively good. In Test $1$, the relative $L^2$ error on $\bar u$ and $\nabla\bar u$ are respectively $0.017$ and $0.019$. In Test $2$, the coincide region is still captured despite the internal distorted cells.

\medskip
\begin{figure}[ht]
	\begin{center}
	\begin{tabular}{c@{\hspace*{1em}}c}
	\includegraphics[width=.48\linewidth]{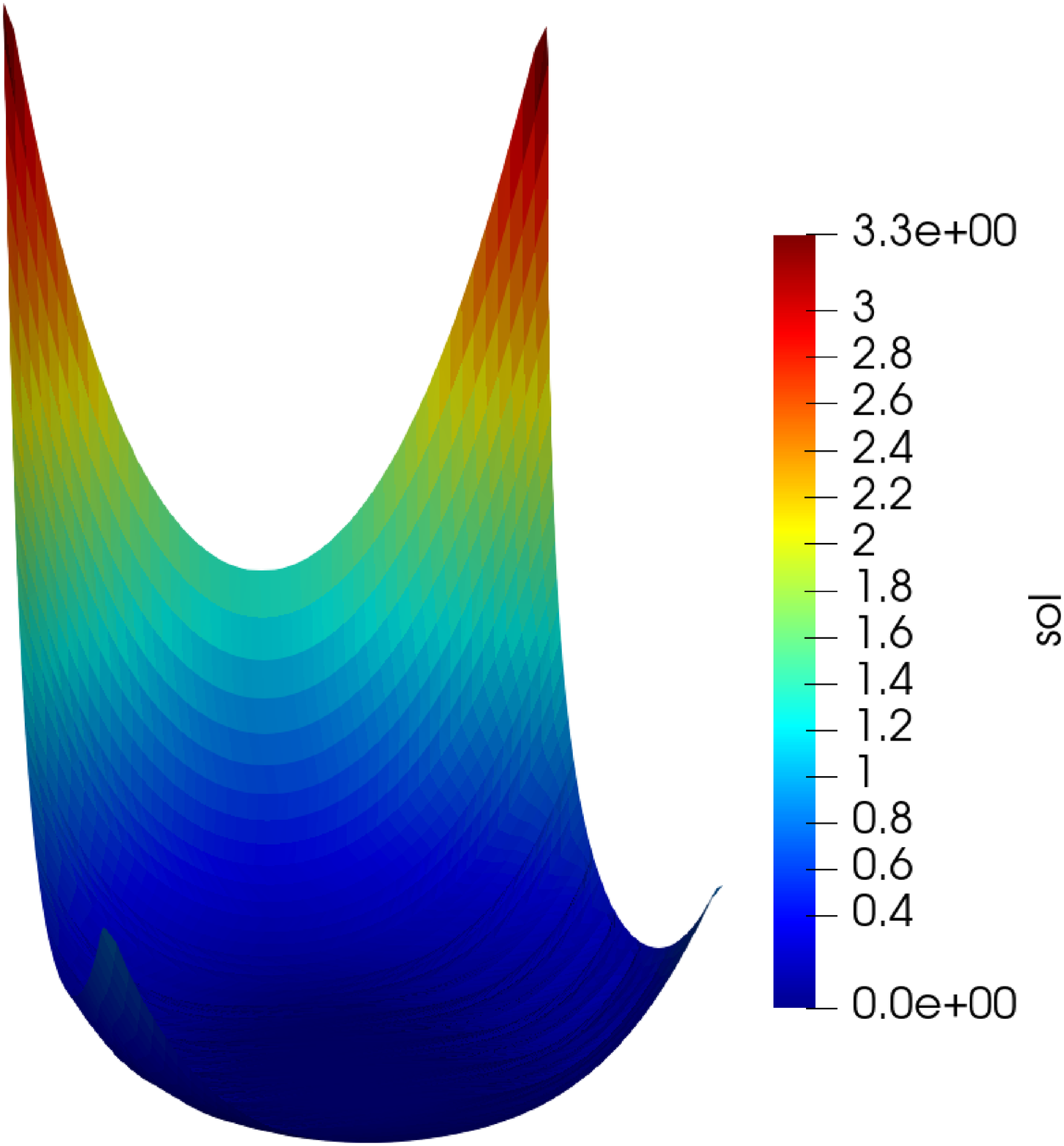} & \includegraphics[width=.48\linewidth]{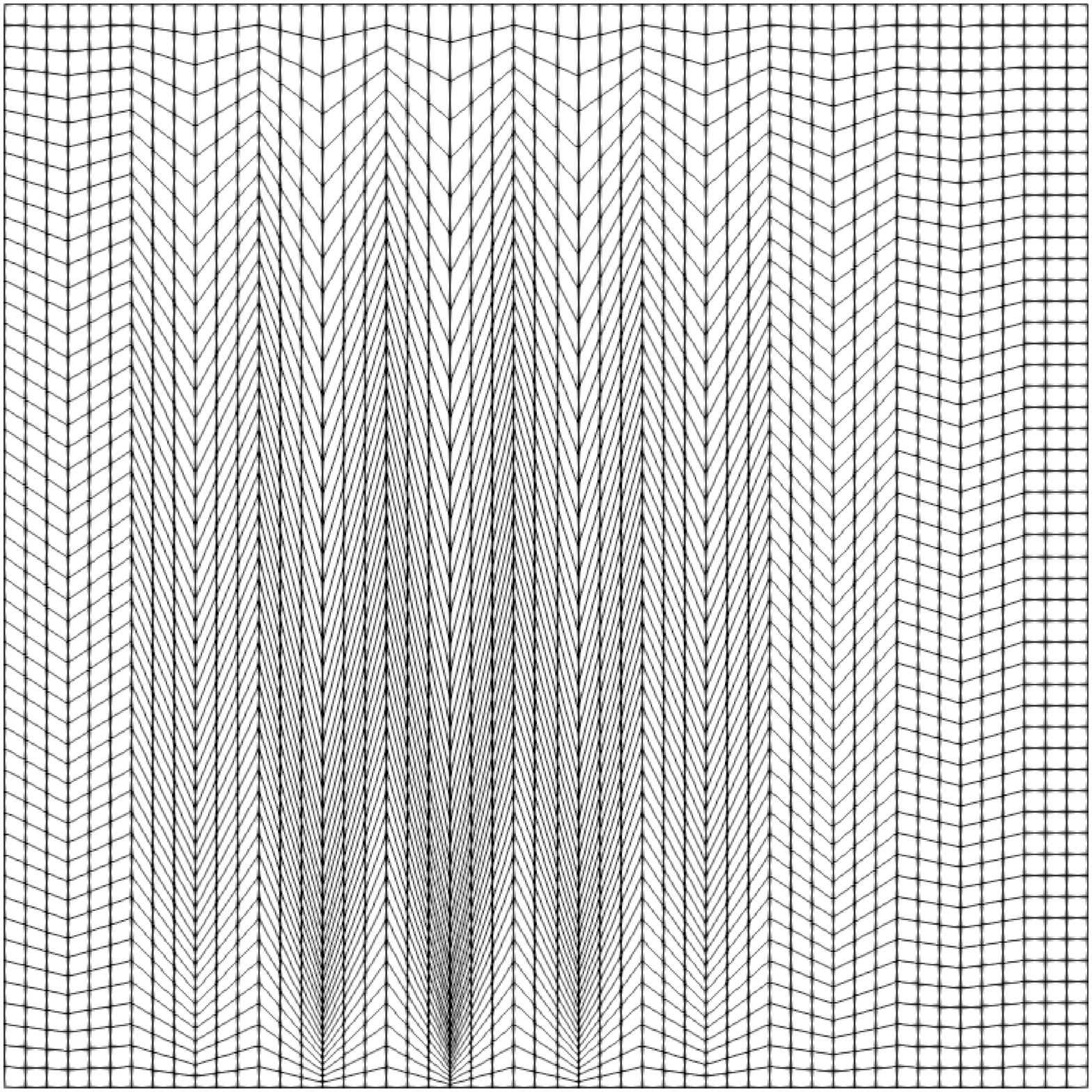}
	\end{tabular}
	\end{center}
	\caption{Test \ref{test-1}. Surface plot of the solution (left) for on a Kershaw mesh (right) at final time ($T=0.25$).}
	\label{Fig-test-1-K-mesh}
\end{figure}

\medskip
\begin{figure}[ht]
	\begin{center}
	\begin{tabular}{c@{\hspace*{1em}}c}
	\includegraphics[width=.48\linewidth]{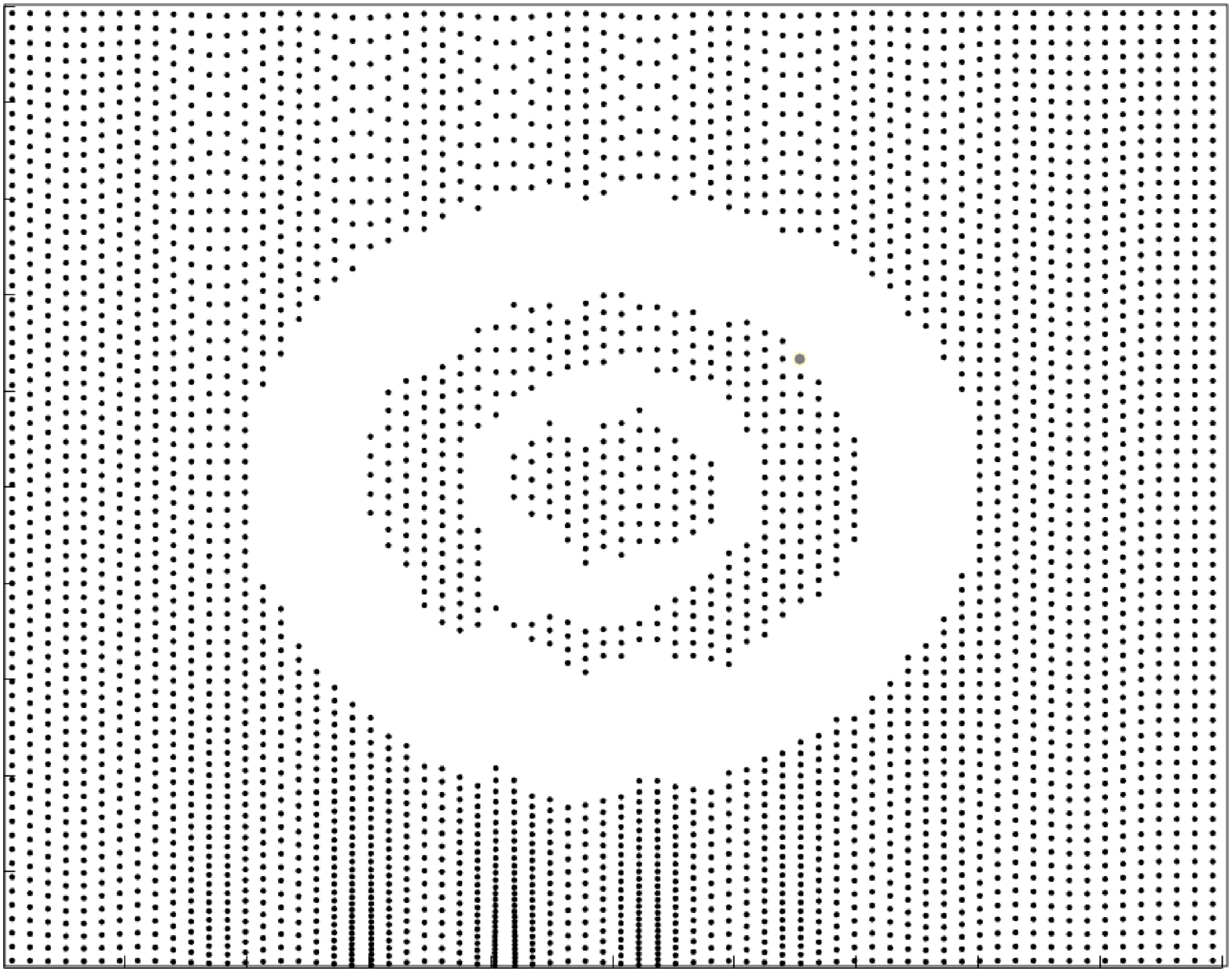} & \includegraphics[width=.48\linewidth]{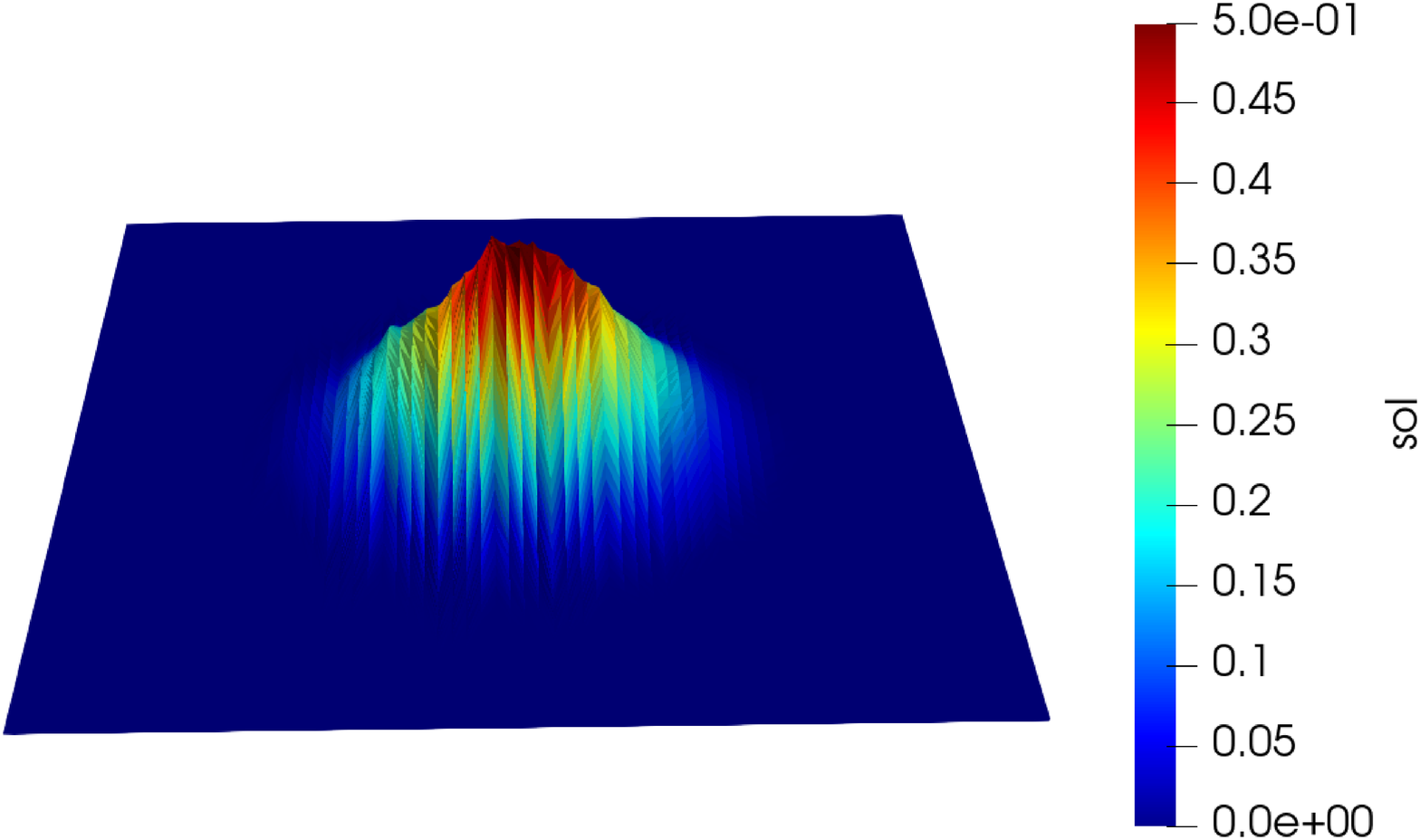}
	\end{tabular}
	\end{center}
	\caption{Test \ref{test-2}. Surface plot of the solution (left) for on a Kershaw mesh (right) at final time ($T=0.25$).}
	\label{Fig-test-2-K-mesh}
\end{figure}


\bibliographystyle{siam}
\bibliography{pvi-ref}

\begin{thebibliography}{10}

\bibitem{AD14}
{\sc Y.~Alnashri and J.~Droniou}, {\em Gradient schemes for the {S}ignorini and
  the obstacle problems, and application to hybrid mimetic mixed methods},
  Computers and Mathematics with Applications, 72 (2016), pp.~2788--2807.

\bibitem{YD-2018}
\leavevmode\vrule height 2pt depth -1.6pt width 23pt, {\em A gradient
  discretization method to analyze numerical schemes for nonlinear variational
  inequalities, application to the seepage problem}, SIAM Journal on Numerical
  Analysis, 56 (2018), pp.~2375--2405.

\bibitem{ex-pvi-soln-2}
{\sc C.~Baiocchi}, {\em Discretization of evolution variational inequalities},
  in Partial differential equations and the calculus of variations, Springer,
  1989, pp.~59--92.

\bibitem{N-pvi-4}
{\sc A.~E. Berger and R.~S. Falk}, {\em An error estimate for the truncation
  method for the solution of parabolic obstacle variational inequalities},
  Mathematics of Computation, 31 (1977), pp.~619--628.

\bibitem{FVM-pvi}
{\sc J.~Berton and R.~Eymard}, {\em Finite volume methods for the valuation of
  american options}, ESAIM: Mathematical Modelling and Numerical Analysis -
  Mod\'elisation Math\'ematique et Analyse Num\'erique, 40 (2006),
  pp.~311--330.

\bibitem{reg-pvi-soln}
{\sc H.~Brezis}, {\em Ope{\'e}rateurs maximaux monotones et semi-groupes de
  contractions dans les espaces de Hilbert}, vol.~5, Elsevier, 1973.

\bibitem{Brezzi-1991}
{\sc F.~Brezzi and M.~Fortin}, {\em Mixed and Hybrid Finite Element Methods},
  Springer-Verlag, Berlin, Heidelberg, 1991.

\bibitem{psi-finance}
{\sc M.~Broadie and J.~Detemple}, {\em The valuation of american options on
  multiple assets}, Mathematical Finance, 7 (1997), pp.~241--286.

\bibitem{PVI-test-2009}
{\sc L.~Brugnano and A.~Sestini}, {\em Numerical solution of obstacle and
  parabolic obstacle problems based on piecewise linear systems}, in AIP
  Conference Proceedings, vol.~1168, AIP, 2009, pp.~746--749.

\bibitem{P7}
{\sc C.~Carstensen and J.~Gwinner}, {\em A theory of discretization for
  nonlinear evolution inequalities applied to parabolic signorini problems},
  Annali di Matematica Pura ed Applicata, 177 (1999), pp.~363--394.

\bibitem{P6}
{\sc S.~Dongyang, G.~Hongbo, et~al.}, {\em A class of crouzeix--raviart type
  nonconforming finite element methods for parabolic variational inequality
  problem with moving grid on anisotropic meshes}, Hokkaido Mathematical
  Journal, 36 (2007), pp.~687--709.

\bibitem{D2017-PPDE}
{\sc J.~Droniou, R.~Eymard, T.~Gallou{\"e}t, C.~Guichard, and R.~Herbin}, {\em
  An error estimate for the approximation of linear parabolic equations by the
  gradient discretization method}, in Finite Volumes for Complex Applications
  VIII-Methods and Theoretical Aspects, J.~Fuhrmann, M.~Ohlberger, and
  C.~Rohde, eds., vol.~199, Springer Proc. Math. Stat., 2017, pp.~371--379.

\bibitem{S1}
{\sc J.~Droniou, R.~Eymard, T.~Gallou\"et, C.~Guichard, and R.~Herbin}, {\em
  The gradient discretisation method}, Mathematics \& Applications, Springer,
  Heidelberg, 2018.
\newblock To appear.

\bibitem{D-2010-SUSHI}
{\sc R.~{Eymard}, T.~{Gallou\"et}, and R.~{Herbin}}, {\em Discretization of
  heterogeneous and anisotropic diffusion problems on general nonconforming
  meshes sushi: a scheme using stabilization and hybrid interfaces}, IMA
  Journal of Numerical Analysis, 30 (2010), pp.~1009--1043.

\bibitem{P5}
{\sc A.~Fetter}, {\em ${L}^\infty$-error estimate for an approximation of a
  parabolic variational inequality}, Numerische Mathematik, 50 (1987),
  pp.~557--565.

\bibitem{G1}
{\sc R.~Glowinski, J.~Lions, and R.~Tremolieres}, {\em Numerical analysis of
  variational inequalities}, North-Holland Publishing Company, 8~ed., 1981.

\bibitem{2nd-FEM-pvi}
{\sc T.~Gudi and P.~Majumder}, {\em Convergence analysis of finite element
  method for a parabolic obstacle problem}, Journal of Computational and
  Applied Mathematics, 357 (2019), pp.~85 -- 102.

\bibitem{Crouzeix-Raviart-pvi}
\leavevmode\vrule height 2pt depth -1.6pt width 23pt, {\em Crouzeix--€"raviart
  finite element approximation for the parabolic obstacle problem},
  Computational Methods in Applied Mathematics, 20 (2020), pp.~273 -- 292.

\bibitem{HH08}
{\sc R.~Herbin and F.~Hubert}, {\em Benchmark on discretization schemes for
  anisotropic diffusion problems on general grids}, in Finite volumes for
  complex applications {V}, ISTE, London, 2008, pp.~659--692.

\bibitem{A-2}
{\sc R.~{Herbin} and E.~{Marchand}}, {\em Finite volume approximation of a
  class of variational inequalities}, IMA Journal of Numerical Analysis, 21
  (2001), pp.~553--585.

\bibitem{1st-FEM-pvi}
{\sc C.~Johnson}, {\em A convergence estimate for an approximation of a
  parabolic variational inequality}, SIAM Journal on Numerical Analysis, 13
  (1976), pp.~599--606.

\bibitem{N-pvi-20}
{\sc C.~Johnson}, {\em A convergence estimate for an approximation of a
  parabolic variational inequality}, SIAM Journal on Numerical Analysis, 13
  (1976), pp.~599--606.

\bibitem{P2}
{\sc J.~Kacur and R.~Van~Keer}, {\em Solution of degenerate parabolic
  variational inequalities with convection}, ESAIM: Mathematical Modelling and
  Numerical Analysis, 37 (2003), pp.~417--431.

\bibitem{ex-pvi-soln}
{\sc J.~Lions and G.~Stampacchia}, {\em Variational inequalities},
  Communications on pure and applied mathematics, 20 (1967), pp.~493--519.

\bibitem{N-21-Louis}
{\sc J.-L. Lions}, {\em Inequalities in mechanics and physics}, Springer, 1976.

\bibitem{P1}
{\sc K.-S. Moon, R.~H. Nochetto, T.~Von~Petersdorff, and C.-s. Zhang}, {\em A
  posteriori error analysis for parabolic variational inequalities}, ESAIM:
  Mathematical Modelling and Numerical Analysis, 41 (2007), pp.~485--511.

\bibitem{Moon-2007}
{\sc {Moon, Kyoung-Sook}, {Nochetto, Ricardo H.}, {von Petersdorff, Tobias},
  and {Zhang, Chen-song}}, {\em A posteriori error analysis for parabolic
  variational inequalities}, ESAIM: M2AN, 41 (2007), pp.~485--511.

\bibitem{P8}
{\sc R.~H. Nochetto, T.~von Petersdorff, and C.-S. Zhang}, {\em A posteriori
  error analysis for a class of integral equations and variational
  inequalities}, Numerische Mathematik, 116 (2010), pp.~519--552.

\bibitem{A4}
{\sc J.~Oden and N.~Kikuchi}, {\em Theory of variational inequalities with
  applications to problems of flow through porous media}, International Journal
  of Engineering Science, 18 (1980), pp.~1173 -- 1284.

\bibitem{P16}
{\sc F.~P{\'e}rez, L.~Ferragut, and J.~M. Casc{\'o}n}, {\em An adaptive method
  for the stefan problem and its application to endoglacial conduits}, Advances
  in Engineering Software, 38 (2007), pp.~423--428.

\bibitem{T1997}
{\sc M.~E. Taylor}, {\em Review: Martin schechter, modern methods in partial
  differential equations, an introduction}, Bull. Amer. Math. Soc. (N.S.), 1
  (1979), pp.~661--667.

\bibitem{reg-pvi-soln2}
{\sc N.~N. Ural'tseva}, {\em H{\"o}lder continuity of gradients of solutions of
  parabolic equations with boundary conditions of signorini type}, in Dokl.
  Akad. Nauk SSSR, vol.~280, 1985, pp.~563--565.

\bibitem{N-pvi-34}
{\sc C.~Vuik}, {\em An {L}2-error estimate for an approximation of the solution
  of a parabolic variational inequality.}, Numerische Mathematik, 57 (1990),
  pp.~453--472.

\bibitem{P10}
{\sc X.~Yang, G.~Wang, and X.~Gu}, {\em Numerical solution for a parabolic
  obstacle problem with nonsmooth initial data}, Numerical Methods for Partial
  Differential Equations, 30 (2014), pp.~1740--1754.

\end{thebibliography}

\end{document}